\documentclass[conference, 10pt]{IEEEtran}

\usepackage{times}

\usepackage{multicol}
\usepackage[hidelinks=true, bookmarks = true]{hyperref} 

\usepackage{amsmath,amsfonts, amsthm, amssymb}
\usepackage{algorithmicx}
\usepackage[Algorithm,ruled]{algorithm}
\usepackage{enumerate}
\usepackage{balance}
\usepackage{commath}

\usepackage{cite}

\usepackage[font={small}]{caption}
\usepackage{float}
\usepackage{capt-of}
\usepackage{cleveref} 
\usepackage{subcaption}
\usepackage{multicol}
\usepackage{graphicx}

\usepackage[font={small}]{caption}
\newtheorem{theorem}{Theorem}

\newtheorem{prop}{Proposition}
\newtheorem{definition}{Definition}
\newtheorem{assumption}{Assumption}

\begin{document}

\title{Feedback Synthesis for Controllable Underactuated Systems using Sequential Second Order Actions}

\author{\authorblockN{Giorgos Mamakoukas} \and
	\authorblockN{Malcolm A. MacIver} \and
	\authorblockN{Todd D. Murphey}
	\thanks{The authors are with the Department of Mechanical Engineering (Mamakoukas, MacIver and Murphey), with the Department of Biomedical Engineering (MacIver), and the Department of Neurobiology (MacIver), Northwestern University, Evanston, IL, 60208 USA. Emails:\,{\tt\small giorgosmamakoukas@u.northwestern.edu, maciver@northwestern.edu, \newline t-murphey@northwestern.edu}}}
\maketitle

\begin{abstract}
This paper derives nonlinear feedback control synthesis for general control affine systems using second-order actions---the needle variations of optimal control---as the basis for choosing each control response to the current state. A second result of the paper is that the method provably exploits the nonlinear controllability of a system by virtue of an explicit dependence of the second-order needle variation on the Lie bracket between vector fields. As a result, each control decision necessarily decreases the objective when the system is nonlinearly controllable using first-order Lie brackets. Simulation results using a differential drive cart, an underactuated kinematic vehicle in three dimensions, and an underactuated dynamic model of an underwater vehicle demonstrate that the method finds control solutions when the first-order analysis is singular. Moreover, the simulated examples demonstrate superior convergence when compared to synthesis based on first-order needle variations. Lastly, the underactuated dynamic underwater vehicle model demonstrates the convergence even in the presence of a velocity field. 
\end{abstract}
\IEEEpeerreviewmaketitle
\section{Introduction}
With many important applications in aerial or underwater missions, systems are underactuated either by design---in order to reduce actuator weight, expenses or energy consumption---or as a result of technical failures. In both cases, it is important to develop control policies that can exploit the nonlinearities of the dynamics, are general enough for this broad class of systems, and easily computable.
Various approaches to nonlinear control range from steering methods using sinusoid controls\cite{MurraySin}, sequential actions of Lie bracket sequences\cite{murray1994book} and backstepping\cite{kokotovic1992joy,	seto1994control} to perturbation methods\cite{junkins1986asymptotic}, sliding mode control (SMC)\cite{perruquetti2002sliding,utkin2013sliding, xu2008sliding}, intelligent\cite{brown1997intelligent, harris1993intelligent} or hybrid\cite{fierro1999hybrid} control and nonlinear model predictive control (NMPC) methods\cite{allgower2004nonlinear}. These schemes have been successful on well-studied examples including, but not limited to, the rolling disk, the kinematic car, wheeling mobile robots, the Snakeboard, surface vessels, quadrotors, and cranes \cite{bullo2000controllability, nonholonomiccrane,escareno2012trajectory,reyhanoglu1996nonlinear,fang2003nonlinear,toussaint2000tracking,bouadi2007sliding,bouadi2007modelling,chen2013adaptive, nakazono2008vibration, shammas2012analytic, morbidi2007sliding, roy2007closed, becker2010motion, kolmanovsky1995developments,boskovic1999intelligent}. 

The aforementioned methods are not ideal in dealing with controllable systems. In the case of perturbations, the applied controls assume a future of control decisions that did not take the disturbance history into account; backstepping is generally ineffective in the presence of control limits and NMPC methods are typically computationally expensive. SMC methods suffer from chattering, which results in high energy consumption and instability risks by virtue of exciting unmodeled high-frequency dynamics \cite{khalil1996noninear}, intelligent control methods are subject to data uncertainties\cite{el2014intelligent}, while other methods are often case-specific and will not hold for the level of generality encountered in robotics. We address this limitation by using needle variations to compute feedback laws for general nonlinear systems affine in control, discussed next.
\subsection{Needle Variations Advantages to Optimal Control}
In this paper, we investigate using needle variation methods to find optimal control for nonlinear controllable systems. Needle variations consider the sensitivity of the cost function to infinitesimal application of controls and synthesize actions that reduce the objective\cite{aseev2014needle,shaikh2007hybrid}. Such control synthesis methods have the advantage of efficiency in terms of computational effort, making them appropriate for online feedback (similar to other model predictive control methods, such as iLQG\cite{todorov2005generalized}, but with the advantage---as shown here---of having provable formal properties over the entire state space). For time evolving objectives, as in the case of trajectory tracking tasks, controls calculated from other methods (such as sinusoids or Lie brackets for nonholonomic integrators) may be rendered ineffective as the target continuously moves to different states. In such cases, needle variation controls have the advantage of computing actions that directly reduce the cost, without depending on future control decisions. However, needle variation methods, to the best of our knowledge, have not yet considered higher than first-order sensitivities of the cost function. 

We demonstrate analytically later in Section II that, by considering second-order needle variations, we obtain variations that explicitly depend on the Lie brackets between vector fields and, as a consequence, the higher-order nonlinearities in the system. Later, in Section III, we show that, for classically studied systems, such as the differential drive cart, this amounts to being able to guarantee that the control approach is \emph{globally} certain to provide descent at every state, despite the conditions of Brockett's theorem \cite{brockett1983asymptotic} on nonexistence of smooth feedback laws for such systems. 

\subsection{Paper Contribution and Structure}
This paper derives the second-order sensitivity of the cost function with respect to infinitesimal duration of inserted control, which we will refer to interchangeably as the second-order mode insertion gradient or mode insertion Hessian (MIH). We relate the MIH expression to controllability analysis by revealing its underlying Lie bracket structure and present a method of using second-order needle variation actions to expand the set of states for which individual actions that guarantee descent of an objective function can be computed. Finally, we compute an analytical solution of controls that uses the first two orders of needle variations. Due to length constraints, the details of some proofs are shortened to an acceptable length, allowing us to include examples demonstrating the method. 

The content is structured as follows. In Section II, we prove that second-order needle variations guarantee control solutions for systems that are nonlinearly controllable using first-order Lie brackets. In Section III, we present an analytical control synthesis method that uses second-order needle actions. In Section IV, we implement the proposed synthesis method and present simulation results on a controllable, underactuated model of a 2D differential drive vehicle, a 3D controllable, underactuated kinematic rigid body and a 3D underactuated dynamic model of an underwater vehicle. 

\section{Needle Variation Controls based on Non-Linear Controllability}
In this section, we relate the controllability of systems to first- and second-order needle variation actions. After presenting the MIH expression, we reveal its dependence on Lie bracket terms between vector fields. Using this connection, we tie the descent property of needle variation actions to the controllability of a system and prove that second-order needle variation controls can produce control solutions for a wider set of the configuration state space than first-order needle variation methods. As a result, we are able to constructively compute, via an analytic solution, control formulas that are guaranteed to provide descent, provided that the system is controllable with first-order Lie brackets. Generalization to higher-order Lie brackets appears to have the same structure, but that analysis is postponed to future work.

\subsection{Second-Order Mode Insertion Gradient}
Consider a system with state $x : \mathbb{R} \mapsto \mathbb{R}^{N \times 1} $ and control $u : \mathbb{R} \mapsto \mathbb{R}^{M \times 1} $ with control-affine dynamics of the form
\begin{align}\label{dynamics}
f(t,x(t), u(t)) = g(t, x(t)) + h(t, x(t)) u(t),
\end{align} 
where $g(t,x(t))$ is the drift vector field. Further consider a time period $[t_o, t_f]$ and control modes described by
\begin{align}\label{Dynamics}
\dot{x}(t) = 
\begin{cases}
f_1 (x(t), v), & t_0\leq t < \tau \\
f_2 (x(t), u), & \tau - \frac{\lambda}{2}\leq t < \tau + \frac{\lambda}{2} \\
f_1 (x(t), v), & \tau + \frac{\lambda}{2} \leq t \le t_f ,
\end{cases}
\end{align}
where $f_1$ and $f_2$ are the dynamics associated with \textit{default} and \textit{inserted} control $v$ and $u$, respectively. Parameters $\lambda$ and $\tau$ are the duration of the inserted dynamics $f_2$ and the switching time between the two modes. Dynamics of the form \eqref{Dynamics} are typically used in optimal control of hybrid systems to optimize the time scheduling of a-priori known modes\cite{egerstedt2006transition}. Here, we use such dynamics to obtain a new control mode $u$ that will optimally perturb the trajectory of any type of system with a needle action\cite{SAC}. Given a cost function $J$ of the form
\begin{equation}\label{cost}
J(x(t)) = \int_{t_o}^{t_f} l_1(x(t)) \mathrm{d}t + m(x(t_f)),
\end{equation}
 where $l_1(x(t))$ is the running cost and $m(x(t))$ the terminal cost, the mode insertion gradient (MIG) is
\medmuskip = 0.5mu
\begin{align}\label{MIG}
\frac{dJ}{d\lambda_+} = \rho^T (f_2 - f_1).
\end{align}
For brevity, the dependencies of variables are dropped. Although not presented here because of the length of the derivation and its similarity to \cite{caldwell2011switching}, a similar analysis shows that, for dynamics that do not directly depend on the control duration, the mode insertion Hessian (MIH) is given by
\medmuskip = 0.5mu
\begin{align}\label{MIH}
\frac{d^2J}{d\lambda_+^2} =& (f_2 - f_1)^T\Omega(f_2-f_1) + \rho^T(D_xf_2 \cdot f_2 + D_xf_1\cdot f_1 \notag\\
&-2 D_xf_1\cdot f_2) - D_x l_1 \cdot (f_2 - f_1),
\end{align}
where $\rho : \mathbb{R} \mapsto \mathbb{R}^{N \times 1}$ and $\Omega : \mathbb{R} \mapsto \mathbb{R}^{N \times N}$ are the first- and second-order adjoint states (costates). These quantities are calculated from the default trajectory and are given by
\begin{align*}
\dot{\rho} &= -{D_xl_1}^T - D_xf_1^T\rho \\
\dot{\Omega} &= -{D_xf_1}^T\Omega - \Omega D_xf_1 - D_x^2l_1 - \sum_{i=1}^N \rho_i D_x^2f_1^i ,
\end{align*}
that are subject to $\rho(t_f) = D_x m(x(t_f))^T$ and $\Omega(t_f) = D_x^2 m(x(t_f))^T$. The superscript $i$ in the dynamics $f_1$ refers to the $i^{th}$ element of the vector and is used to avoid confusion against default and inserted dynamics $f_1$ and $f_2$, respectively.
\subsection{Dependence of Second Order Needle Variations on Lie Bracket Structure}
The Lie bracket of two vectors $f(x)$, and $g(x)$ is
\begin{align*}
[f, g](x) = \frac{\partial g}{\partial x} f(x) - \frac{\partial f}{\partial x}g(x), 
\end{align*}
which generates a control vector that points in the direction of the net infinitesimal change in
state $x$ created by infinitesimal noncommutative flow $\phi_\epsilon^f\,\circ\,\phi_\epsilon^g\, \circ\,\phi_\epsilon^{-f}\,\circ\,\phi_\epsilon^{-g}\,\circ\,x_0$, where $\phi_\epsilon^f$ is the flow along a vector field $f$ for time $\epsilon$\cite{murray1994book, jakubczyk2001introduction}. Lie brackets are most commonly used for their connection to controllability\cite{rashevsky1938connecting,Chow1940}, but here they will show up in the expression describing the second-order needle variation. 

We relate second-order needle variation actions to Lie brackets in order to provide controls that are conditional on the nonlinear controllability of a system. Let $h_i : \mathbb{R} \mapsto \mathbb{R}^{N \times 1}$ denote the column control vectors that compose $h : \mathbb{R} \mapsto \mathbb{R}^{N \times M}$ in \eqref{dynamics} and $u_i \in \mathbb{R}$ be the individual control inputs. Then, we can express dynamics as
\begin{align*}
f = g + \sum_i^M h_iu_i.
\end{align*}
and, for default control $v=0$, we can re-write the MIH as
\medmuskip = 0.3mu
\begin{align*}
\frac{d^2J}{d\lambda_+^2}=&\big(\sum_{i=1}^M h_iu_i\big)^T\,\Omega\sum_{j=1}^Mh_ju_j + \rho^T\Big(\sum_{i=1}^M(D_xh_iu_i)\cdot\,g \\
&-D_xg\cdot(h_iu_i)+\sum_{i=1}^MD_xh_iu_i\sum_{i=1}^Mh_iu_i\Big)-D_xl_1\sum_{i=1}^Mh_iu_i.
\end{align*}
Splitting the sum expression into diagonal ($i=j$) and off-diagonal ($i\ne j$) elements, and by adding and subtracting $2\sum_{i}^M\sum_{j=1}^{i-1}(D_xh_iu_i)(h_ju_j)$, we can write
\begin{align*}
\sum_{i=1}^MD_xh_iu_i\sum_{i=1}^Mh_iu_i =& \sum_{i}^M\sum_{j=1}^{i-1}[h_i,h_j]u_iu_j\\&
+ 2\sum_{i}^M\sum_{j=1}^{i-1}(D_xh_iu_i)(h_ju_j) \\
&+ \sum_{i=j=1}^M(D_xh_iu_i)(h_iu_i).
\end{align*}
Then, we can express the MIH as
\begin{align*}
\frac{d^2J}{d\lambda_+^2} =& \sum_{i=1}^M \sum_{j=1}^M u_i u_j h_i^T \Omega h_j + \rho^T\Big(\sum_{i=2}^M \sum_{j=1}^{i-1} [h_i,h_j] u_i u_j \notag \\
&+ 2 \sum_{i=2}^M\sum_{j=1}^{i-1} (D_x h_i)h_j u_iu_j + \sum_{i=1}^{M}(D_x h_i)h_iu_iu_i \notag \\
&+\sum_{i=1}^M [g, h_i] u_i\Big) - D_x l(\sum_{i=1}^M h_iu_i).
\end{align*}
The expression contains Lie bracket terms of the control vectors that appear in the system dynamics, indicating that second-order needle variations consider higher-order nonlinearities. By associating the MIH to Lie brackets, we next prove that second-order needle variation actions can guarantee decrease of the objective for certain types of controllable systems. 

\subsection{Existence of Control Solutions with First- and Second-Order Mode Insertion Gradients}
	In this section, we prove that the first two orders of the mode insertion gradient can be used to guarantee controls that reduce objectives of the form \eqref{cost} for systems that are controllable with first-order Lie brackets. The analysis is applicable to optimization problems that satisfy the following assumptions.
\begin{assumption}\label{as:1}
	The vector elements of dynamics $f_1$ and $f_2$ are real, bounded, $\mathcal{C}^2$ in $x$, and $\mathcal{C}^0$ in $u$ and $t$.
\end{assumption}
\begin{assumption}\label{as:2}
	The incremental cost $l_1(x)$ is real, bounded, and $\mathcal{C}^2$ in $x$. The terminal cost $m(x(t_f))$ is real and twice differentiable with respect to $x(t_f)$.
\end{assumption}
\begin{assumption}\label{as:3}
	Default and inserted controls $v$ and $u$ are real, bounded, and $\mathcal{C}^0$ in $t$.
\end{assumption}
Under Assumptions \ref{as:1}-\ref{as:3}, the MIG and MIH expressions are well-defined. Then, as we show next, there are control actions that can improve any objective that is not a local optimizer. 
\begin{definition}
	A local optimizer of the cost function $\eqref{cost}$ is given by a set $(x^*, u^*)$ if and only if the set describes a trajectory that corresponds to an objective function $J(x^*(t))$ for which $D_xJ(x^*(t))= 0$.
\end{definition}
\begin{prop}\label{nonzerorho}
	Consider a set $(x, v)$ that describes the state and default control of \eqref{Dynamics}. If $(x, v) \ne (x^*, v^*)$, then the first-order adjoint $\rho$ is a non-zero vector.
\end{prop}
\begin{proof}
	Using \eqref{cost}, 
	\begin{align*}
	x \ne x^* &\Rightarrow D_xJ(x(t)) \ne 0 \\
	&\Rightarrow \int_{t_0}^{t_f} D_xl_1(x(t)) \mathrm{d}t + D_xm(x(t_f)) \ne 0 \\
	&\Rightarrow \int_{t_0}^{t_f} D_xl_1(x(t)) \mathrm{d}t \ne 0~\text{OR}~D_xm(x(t_f)) \ne 0 \\
	&\Rightarrow D_xl_1(x(t)) \ne 0~\text{OR}~D_xm(x(t_f)) \ne 0\\
	&\Rightarrow \dot{\rho} \ne 0\lor \rho(t_f) \ne 0.
	\end{align*}
	Therefore, if $x \ne x^*$, then $\exists~ t\in[t_0,t_f]$ such that $\rho \ne 0$.
\end{proof}
\begin{prop}\label{AdjVec}
	Consider dynamics given by \eqref{Dynamics} and a pair of state and control $(x, v) \ne (x^*, v^*)$ such that $\frac{dJ}{d\lambda_+} = 0 ~ \forall ~ u~\in \mathbb{R}^M$ and $\forall ~ t \in [t_o, t_f]$. Then, the first-order adjoint $\rho$ is orthogonal to all control vectors $h_i$. 
\end{prop}
\begin{proof}
	The linear combination of the elements of a vector $x$ is always zero if and only if $x$ is the zero vector. Given that, rewrite \eqref{MIG} as
	\begin{align*}
	\frac{d J}{d \lambda_+}=0 &\Rightarrow \rho^T \sum_i^M h_i (u_i - v_i) = 0 \\
	&\Rightarrow \sum_i^M k_i w_i = 0 ~ \forall ~ w_i,
	\end{align*}
	where $w_i = (u_i - v_i)$ and $k_i = \rho^Th_i \in \mathbb{R}$. The linear combination of the elements of $k$ is zero for any $w_i$, which means $k$ must be the zero vector. By Proposition \ref{nonzerorho}, $\rho \ne 0$ for a non-optimizer pair of state and control and, as a result, $\rho^T h_i~=~0~\forall~i\in[1,M]$.
\end{proof}
\begin{prop}\label{AdjLie}
	Consider dynamics given by \eqref{Dynamics} and a pair of state and control $(x, v) \ne (x^*, v^*)$ such that $\frac{dJ}{d\lambda_+} = 0 ~ \forall ~ u \in \mathbb{R}^M$ and $\forall ~ t \in [t_o, t_f]$. Further assume that the control vectors $h_i$ and their Lie Bracket terms $[h_i, h_j]$ span the state space $\mathbb{R}^N$. Then, there exist $i$ and $j$ such that $\rho^T [h_i, h_j] \ne 0$. 
\end{prop}
\begin{proof}
	A set of vectors $S = (v_1, \dots, v_M)$ is linearly independent if and only if every vector $r\in span(S)$ can be uniquely written as a linear combination of $(v_1, \dots, v_M)$. The control vectors and their Lie Brackets span the $\mathbb{R}^N$ space. On that assumption, it follows that any N-dimensional vector can be expressed in terms of the control vectors and their Lie Brackets. The first-order adjoint is an $N$-dimensional vector, which is non-zero for a non-optimizer pair of $x,v$ by Preposition \ref{nonzerorho}. Therefore, it can be expressed as
	\begin{equation}\label{RHO}
	\rho = c_1 h_1 + \dots + c_M h_M + \sum_{i\ne j}^M [h_i, h_j] \ne 0.
	\end{equation}
	Given that $\frac{dJ}{d\lambda_+}=0$, and by Proposition \ref{AdjVec}, $\rho$ is orthogonal to all control vectors $h_i$ (which also implies that the control vectors $h_i$ do not span $\mathbb{R}^N$). Then, left-multiplying \eqref{RHO} by $\rho^T$ yields
	\begin{align*}
	\rho^T \rho = \sum_{i\ne j}^M \rho^T[h_i, h_j] \ne 0.
	\end{align*}It follows that there is at least one Lie bracket term $[h_i, h_j]$ that is not orthogonal to the costate $\rho$.
\end{proof}
\begin{prop}\label{NegGrad}
		Consider dynamics given by \eqref{Dynamics} and a trajectory described by state and control $(x, v)$. If $(x, v) \ne (x^*, v^*)$, then there are always control solutions $u \in \mathbb{R}^M$ such that $\frac{dJ}{d\lambda_+} \le 0$ for some $t \in [t_o, t_f].$ 
\end{prop}
\begin{proof}
	Using dynamics of the form in \eqref{dynamics}, the expression of the mode insertion gradient can be written as
	\begin{align*}
	\frac{dJ}{d\lambda_+} = \rho^T(f_2 - f_1) = \rho^T\big(h(u-v)\big).
	\end{align*}
	By Proposition \ref{nonzerorho}, $\rho \ne 0$ for a non-optimizer trajectory. Given controls $u$ and $v$ that generate a positive mode insertion gradient, there always exist control $u'$ such that the mode insertion gradient is negative, i.e. $u'-v = - (u-v)$. 
	The mode insertion gradient is zero for all $u\in\mathbb{R}^M$ if and only if the costate vector is orthogonal to each control vector $h_i$\footnote{If the control vectors span the state space $\mathbb{R}^N$, the costate vector $\rho \in \mathbb{R}^N$ cannot be orthogonal to each of them. Therefore, for first-order controllable (fully actuated) systems, there always exist controls for which the cost can be reduced to first order.}.
\end{proof}
First-order needle variation methods are singular when the mode insertion gradient is zero. When that is true, the second-order mode insertion gradient is guaranteed to be negative for systems that are controllable with first-order Lie Brackets, which in turn implies that a control solution can be found with second-order needle variation methods.
\begin{prop}\label{nonkin}
Consider dynamics given by \eqref{Dynamics} and a trajectory described by state and control $(x, v) \ne (x^*, v^*)$ such that $\frac{dJ}{d\lambda_+} = 0$ for all $u \in \mathbb{R}^M$ and $t \in [t_o, t_f]$. If the control vectors $h_i$ and the Lie brackets $[h_i, h_j]$ and $[g, h_i]$ span the state space ($\mathbb{R}^N$), then there always exist control solutions $u\in \mathbb{R}^M$ such that $\frac{d^2J}{d\lambda_+^2} < 0$. 
\end{prop}
\begin{proof}
	Let $k \in [1, M]$ be an index chosen such that $[h_i, h_k]$ for some $i\in[1,M]\setminus \{k\}$\footnote{The notation $\setminus$ indicates that the element $k$ is subtracted from the set $[1,M]$.} is a vector that is linearly independent of all control vectors $h_i~\forall~i\in[1,M]$. The proof then considers controls such that $u_j~=~v_i~\forall~ j,i\ne k$ and $v_k = 0$ and expresses the MIH expression \eqref{MIH} as 
	\begin{align*}
	\frac{d^2J}{d\lambda_+^2} = u^T \mathcal{G} u - u_k((D_xl_1) h_k - \rho^T[g, h_k]) ,
	\end{align*} 
	where $\mathcal{G}_{ij} = 0~\forall~i, j \in[1,M]\setminus \{k\}$, $\mathcal{G}_{ik} = \mathcal{G}_{ki} = \frac{1}{2}[h_i, h_k]$, and $\mathcal{G}_{kk} = h_k^T\Omega h_k + \rho^TD_xh_k\cdot h_k$. 
	The matrix $\mathcal{G}$ is shown to be either indefinite or negative semidefinite if there exists a Lie bracket term $[h_i, h_k]$ such that $\rho^T [h_i, h_k] \ne 0$. If, on the other hand, $\rho^T [h_i, h_k] = 0$, by reasoning that is similar to Proposition \ref{AdjLie}, there is at least a Lie bracket $[g, h_k] \ne 0$ and the MIH expression reduces to a quadratic in $u_k$. In either case, it then becomes straightforward to show that there exist controls for which the MIH expression is negative. 
\end{proof}
\begin{theorem}\label{Theorem}
	Consider dynamics given by \eqref{Dynamics} and a trajectory described by state and control $(x, v) \ne (x^*, v^*)$. If the control vectors $h_i$ and the Lie brackets $[h_i, h_j]$ and $[g, h_i]$ span the state space $(\mathbb{R}^N)$, then there always exists a control vector $u \in \mathbb{R}^M$ and a duration $\lambda$ such that the cost function \eqref{cost} can be reduced. 
\end{theorem}
\begin{proof}
	The local change of the cost function \eqref{cost} due to inserted control $u$ of duration $\lambda$ can be approximated with a Taylor series expansion
	\begin{align*}
	J(\lambda) - J(0) \approx \lambda \frac{dJ}{d\lambda_+} + \frac{\lambda^2}{2} \frac{d^2J}{d\lambda_+^2}.
	\end{align*}
	By Propositions \ref{NegGrad} and \ref{nonkin}, either 1) $\frac{dJ}{d\lambda_+} <0$ or 2) $\frac{dJ}{d\lambda_+} = 0$ and $\frac{d^2J}{d\lambda_+^2}<0$. Therefore, there always exist controls that reduce the cost function \eqref{cost} to first or second order.
\end{proof}

\section{Control Synthesis}
In this section, we present an analytical solution of first- and second-order needle variation controls that reduce the cost function \eqref{cost} to second order. We then describe the algorithmic steps of the feedback scheme used in the simulation results of this paper. 

\subsection{Analytical Solution for Second Order Actions}
For underactuated systems, there are states at which $\rho$ is orthogonal to the control vectors $h_i$ (see Proposition \ref{NegGrad}). At these states, control calculations based only on first-order sensitivities fail, while controls based on second-order information still decrease the objective provided that the control vectors and their Lie brackets span the state space (see Theorem \ref{Theorem}). We use this property to compute an analytical synthesis method that expands the set of states for which individual actions that guarantee descent of an objective function can be computed. 

Consider the Taylor series expansion of the cost around control duration $\lambda$. Given the expressions of the first- and second-order mode insertion gradients, we can write the cost function \eqref{cost} as a Taylor series expansion around the infinitesimal duration $\lambda$ of inserted control $u$:
\begin{align}
J(\lambda) & \approx J(0) + \lambda \frac{dJ}{d\lambda_+} + \frac{\lambda^2}{2} \frac{d^2J}{d\lambda_+^2} \notag.
\end{align}
The first- and second-order mode insertion gradients used in the expression are functions of the inserted control $u(t)$ in \eqref{Dynamics}. For a fixed $\lambda$, we can minimize the function using Newton's Method to update the control actions. Control solutions that minimize the Taylor expansion of the cost will have the form
\begin{align}\label{Taylor}
u^{\>*}(t)=& \underset{u}{\operatorname{argmin}} ~J(0) + \lambda \frac{dJ}{d\lambda_+} + \frac{\lambda^2}{2}\frac{d^2J}{d\lambda_+^2} +\frac{1}{2} \lVert u \rVert^2_R,
\end{align}
where the MIH has both linear and quadratic terms in $u(t)$. The time dependence of the control $u$ is purposefully used here to emphasize that control solutions are functions of time $t$. Using the G\^ateaux derivative, we computed the minimizer of \eqref{Taylor} to be
\begin{align}\label{optcon}
u^{\>*}(t)=&[\frac{\lambda^2}{2}\,\Gamma + R] ^{-1} \, [\frac{\lambda^2}{2}\,\Delta + \lambda (-h^T\rho)],
\end{align}
where $\Delta: \mathbb{R} \mapsto \mathbb{R}^{M\times1}$ and $\Gamma: \mathbb{R}\mapsto \mathbb{R}^{M\times M}$ are respectively the first- and second-order derivatives of $d^2J/d\lambda_+^2$ with respect to the control $u$ and are given by
\begin{align*}
\Delta\triangleq& \Big[\big[h^T \big(\Omega^T + \Omega\big)h +
2 h^T \cdot(\sum_{k=1}^{n} (D_xh_k)\rho_{k})^T\big]v
\notag\\
& + {(D_xg \cdot{h})}^{T} \rho
-
(\sum_{k=1}^{n} (D_xh_k)\rho_{k}) \cdot g 
+
h^T D_xl^T\Big]\\
\Gamma \triangleq& [h^T \big(\Omega^T + \Omega\big)h + 
h^T \cdot (\sum_{k=1}^{n} (D_xh_k)\rho_{k})^T+ 
\sum_{k=1}^{n} (D_xh_k)\rho_{k}\cdot h]^T.
\end{align*}
The parameter $R$ denotes a metric on control effort.

The existence of control solutions in \eqref{optcon} depend on the inversion of the Hessian $H = \frac{\lambda^2}{2}\,\Gamma + R$. To ensure H is positive definite, we implement a spectral decomposition on the Hessian $H~=~VDV^{-1}$, where matrices $V$ and $D$ contain the eigenvectors and eigenvalues of $H$, respectively. We replace all elements of the diagonal matrix $D$ that are smaller than $\epsilon$ with $\epsilon$
to obtain $\bar{D}$ and replace $H$ with $\bar{H} = V\bar{D}V^{-1}$ in \eqref{optcon}. We prefer the spectral decomposition approach to the Levenberg-Marquardt method ($\bar{H} = H + \kappa I \succ 0$), because the latter affects all eigenvalues of the Hessian and further distorts the second-order information. At saddle points, we set the control equal to the eigenvector of $H$ that corresponds to the most negative eigenvalue in order to descend along the direction of most negative curvature\cite{murray2010newton,schnabel1990new, boyd2004convex, nocedal2006sequential}. 

This synthesis technique provides controls at time $t$ that guarantee to reduce the cost function \eqref{cost} for systems that are controllable using first-order Lie brackets. Control solutions are computed solely by forward simulating the state over a time horizon $T$ and backward simulating the first- and second-order costates $\rho$ and $\Omega$. As we see next, this leads to a very natural, and easily implementable, algorithm for applying cost-based feedback. 

\subsection{Algorithmic Description of Control Synthesis Method}
The proposed second-order analytical controls presented in \eqref{optcon} are implemented in a series of steps shown in Algorithm \ref{algorithm}.
\begin{algorithm}
	\begin{enumerate}[{1.}]
		\item Simulate states and costates with default dynamics $f_1$ over a time horizon $T$
		\item Compute optimal needle variation controls
		\item Saturate controls
		\item Find the insertion time that corresponds to the most negative mode insertion gradient 
		\item Use a line search to find control duration that ensures reduction of the cost function \eqref{cost}
	\end{enumerate} 
	\caption{}
	\label{algorithm}
\end{algorithm}
We compare first- and second-order needle variation actions by implementing different controls in Step 2 of Algorithm \ref{algorithm}. For the first-order case, we implement controls 
that are the solution to a minimization problem of the first-order sensitivity of the cost function \eqref{cost} and the control effort
\begin{align}\label{optimalu}
u^*(t) &= \min\limits_{u} ~~ \frac{1}{2} (\frac{dJ_1}{d\lambda^+_i}-\alpha_d)^2+\frac{1}{2} \lVert u \rVert^2_R \notag\\ 
&= (\Lambda + R^T)^{-1}(\Lambda v + h^T\rho \alpha_d),
\end{align}
where $\Lambda \triangleq h^T\rho\rho^Th$ and $\alpha_d \in \mathbb{R}^- $ expresses the desired value of the mode insertion gradient term (see, for example, \cite{mamakoukas2016}). Typically, $\alpha_d = \gamma J_o$, where $J_o$ is the cost function \eqref{cost} computed using default dynamics $f_1$. For second-order needle variation actions, we compute controls using \eqref{optcon}.

\subsection{Comparison to Alternative Optimization Approaches}
Algorithm \ref{algorithm} differs from controllers that compute control sequences over the entire time horizon in order to locally minimize the cost function. Rather, the proposed scheme utilizes the time-evolving sensitivity of the objective to infinitesimal switched dynamics and searches in a one-dimensional space for a finite duration of a single action that will optimally improve the cost. It does so using a closed-form expression and, as a result, it avoids the expensive iterative computational search in high-dimensional spaces, while it may still get closer to the optimizer with one iterate. 
\\\indent First-order needle variation solution are shown in \eqref{optimalu} to exist globally, demonstrate a larger region of attraction and have a less complicated representation on Lie Groups\cite{taosha}. These traits naturally transfer to second-order needle controls \eqref{optcon} that also contain the first-order information that is present in \eqref{optimalu}. In addition, as this paper demonstrates, the suggested second-order needle variation controller has formal guarantees of descent for systems that are controllable with first-order Lie brackets, which---to the best of our knowledge---is not provided by any alternative method. Given these benefits, the authors propose second-order needle variation actions as a complement to existing approaches for time-sensitive robotic applications that may be subject to large initial error, Euler angle singularities, or fast-evolving (and uncertain) objectives.

Next, we implement Algorithm \ref{algorithm} using first or second-order needle variation controls (shown in \eqref{optimalu} and \eqref{optcon}, respectively) to compare them in terms of convergence success on various underactuated systems. 
\section{Simulation Results}
The proposed synthesis method is implemented on three underactuated examples, the differential drive cart, a 3D kinematic rigid body and a dynamic model of an underwater vehicle. The kinematic systems of a 2D differential drive and a 3D rigid body are controllable using first-order Lie brackets of the vector fields and help verify Theorem \ref{Theorem}. The underactuated dynamic model of a 3D rigid body serves to compare controls in \eqref{optcon} and \eqref{optimalu} in a more sophisticated environment. In all simulation results, we start with default control $v = 0$ and an objective function of the form
\begin{align*}
J(x(t)) = \frac{1}{2}\int_{t_o}^{t_f} \lVert \vec{x}(t)-\vec{x}_d (t) \rVert^2_Q dt+\frac{1}{2}\lVert \vec{x}(t_f)-\vec{x}_d(t_f)\rVert^2_{P_1},
\end{align*}
where $\vec{x}_d$ is the desired state-trajectory, and $Q=Q^T \geq 0$, $P_1=P_1^T \geq 0$ are metrics on state error.
\subsection{2D Kinematic Differential Drive}
The differential drive system demonstrates that controls shown in \eqref{optimalu} that are based only on the first-order sensitivity of the cost function \eqref{cost} can be insufficient for controllable systems, contrary to controls shown in \eqref{optcon} that guarantee decrease of the objective for systems that are controllable using first-order Lie brackets (see Theorem \ref{Theorem}). The system states are its coordinates and orientation, given by $s = [x, y, \theta]^T$, with kinematic ($g=0$) dynamics
\begin{align*}
f = r\begin{bmatrix} cos(\theta) & cos(\theta) \\
					 sin(\theta) & sin(\theta) \\
					\frac{1}{L} & -\frac{1}{L}\end{bmatrix}
					\begin{bmatrix}u_R \\ u_L \end{bmatrix},
\end{align*}
where $r = 3.6$~cm, $L = 25.8$~cm denote the wheel radius and the distance between them, and $u_R$, $u_L$ are the right and left wheel control angular velocities, respectively (these parameter values match the specifications of the iRobot Roomba).

The control vectors $h_1$, $h_2$ and their Lie bracket term $[h_1, h_2] = 2\frac{r^2}{L}\big[-sin(\theta),-cos(\theta)\big]^T$ span the state space ($\mathbb{R}^3$). Therefore, from Theorem \ref{Theorem}, there always exist controls that reduce the cost to first or second order. 

Fig.~\ref{Differential Drive} demonstrates how different first- and second-order needle variation actions perform on reaching a nearby target. Actions based on first-order needle variations \eqref{optimalu} do not generate solutions that turn the vehicle, but rather drive it straight until the orthogonal displacement between the system and the target location is minimized. Actions based on second-order needle variations \eqref{optcon}, on the other hand, converge successfully. 

\begin{figure}[] 
	\begin{subfigure}[b]{0.45\linewidth}
		\centering
		\includegraphics[width=\linewidth,height = 0.15\textheight]{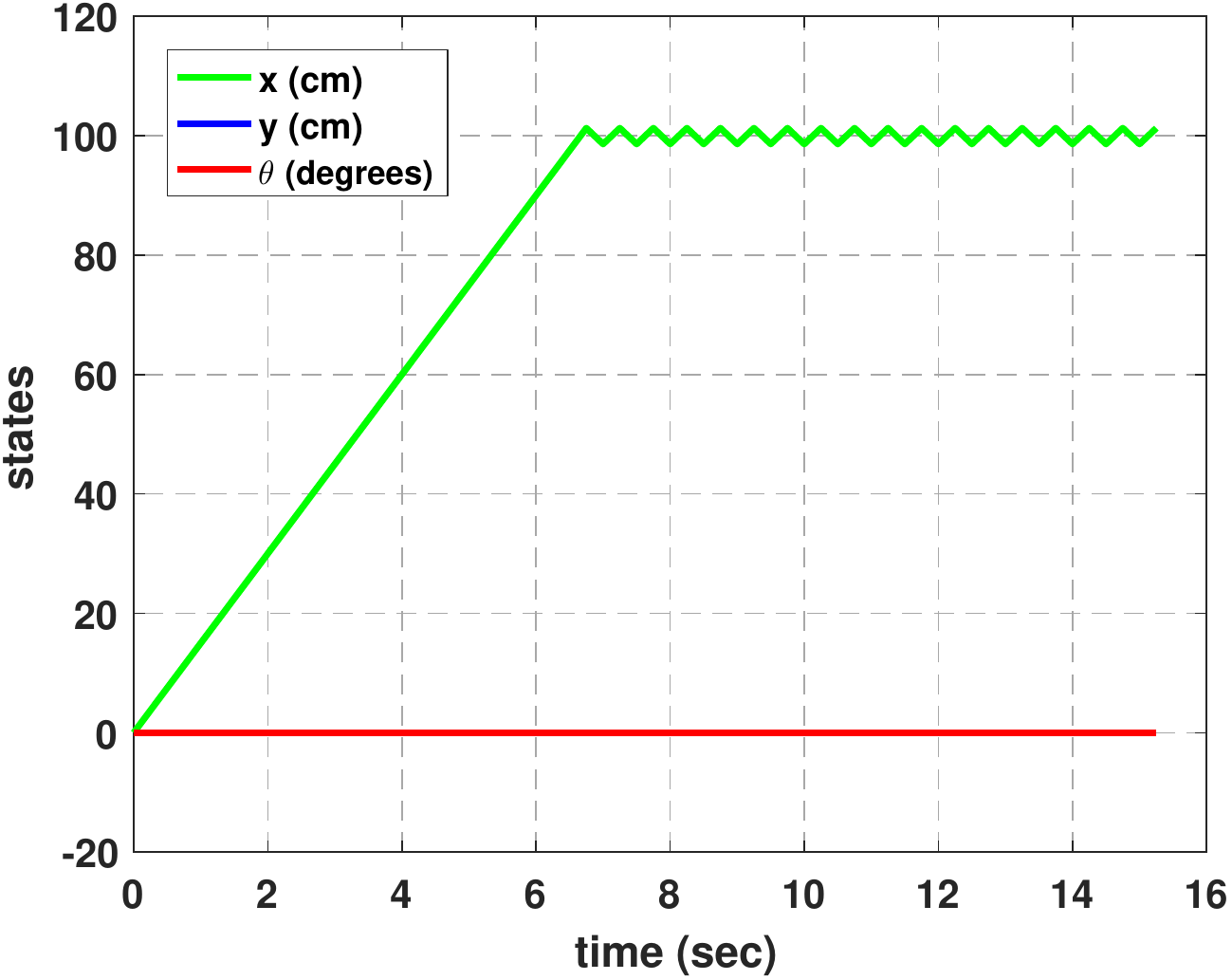} 
		\caption{} 
		\vspace{4ex}
	\end{subfigure}
	\hfill%
	\begin{subfigure}[b]{0.45\linewidth}
		\centering
		\includegraphics[width=\linewidth,height = 0.15\textheight]{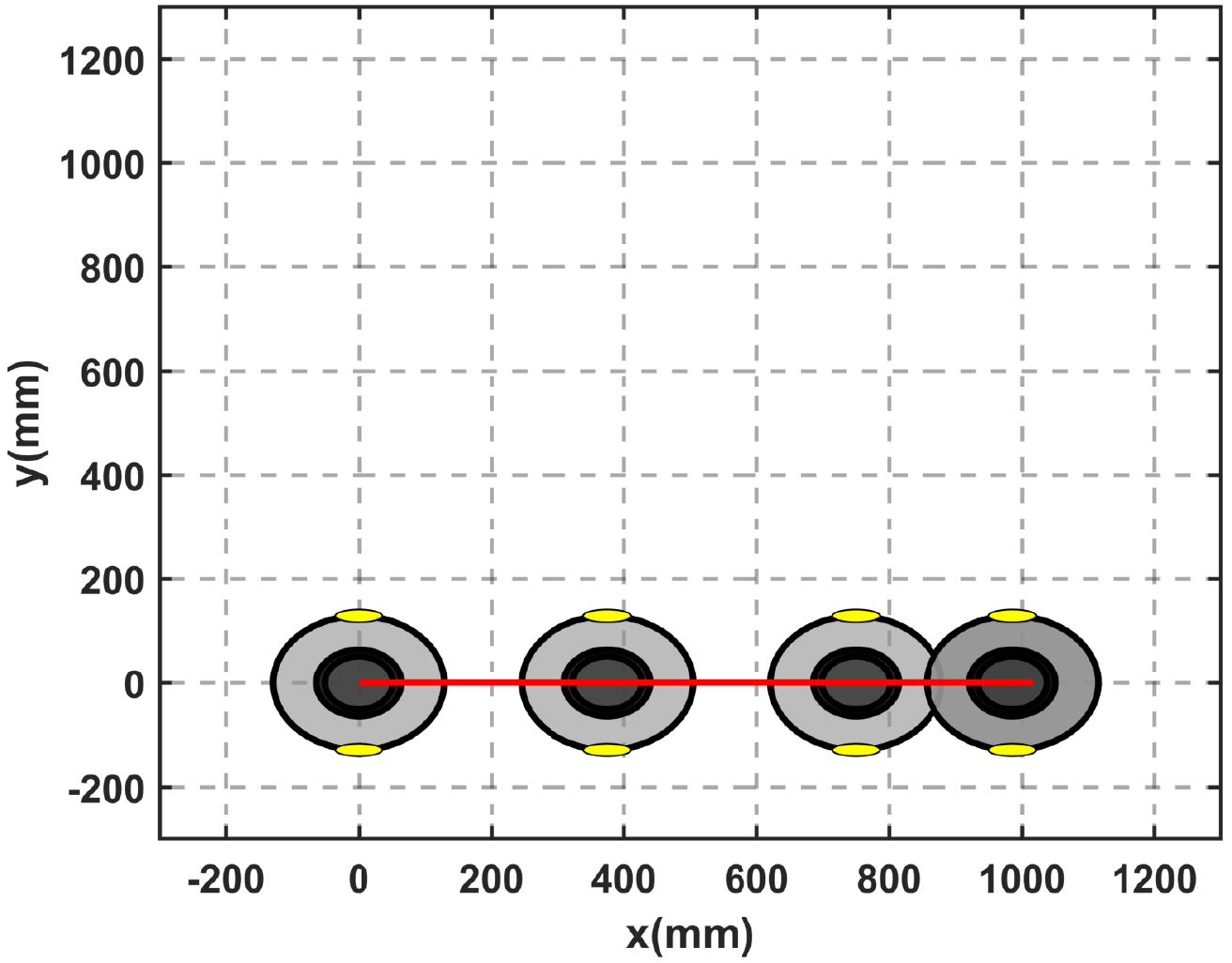} 
		\caption{} 
		\label{fig7:b} 
		\vspace{4ex}
	\end{subfigure} \hfill
	\begin{subfigure}[b]{0.45\linewidth}
		\centering
		\includegraphics[width=\linewidth, height = 0.15\textheight]{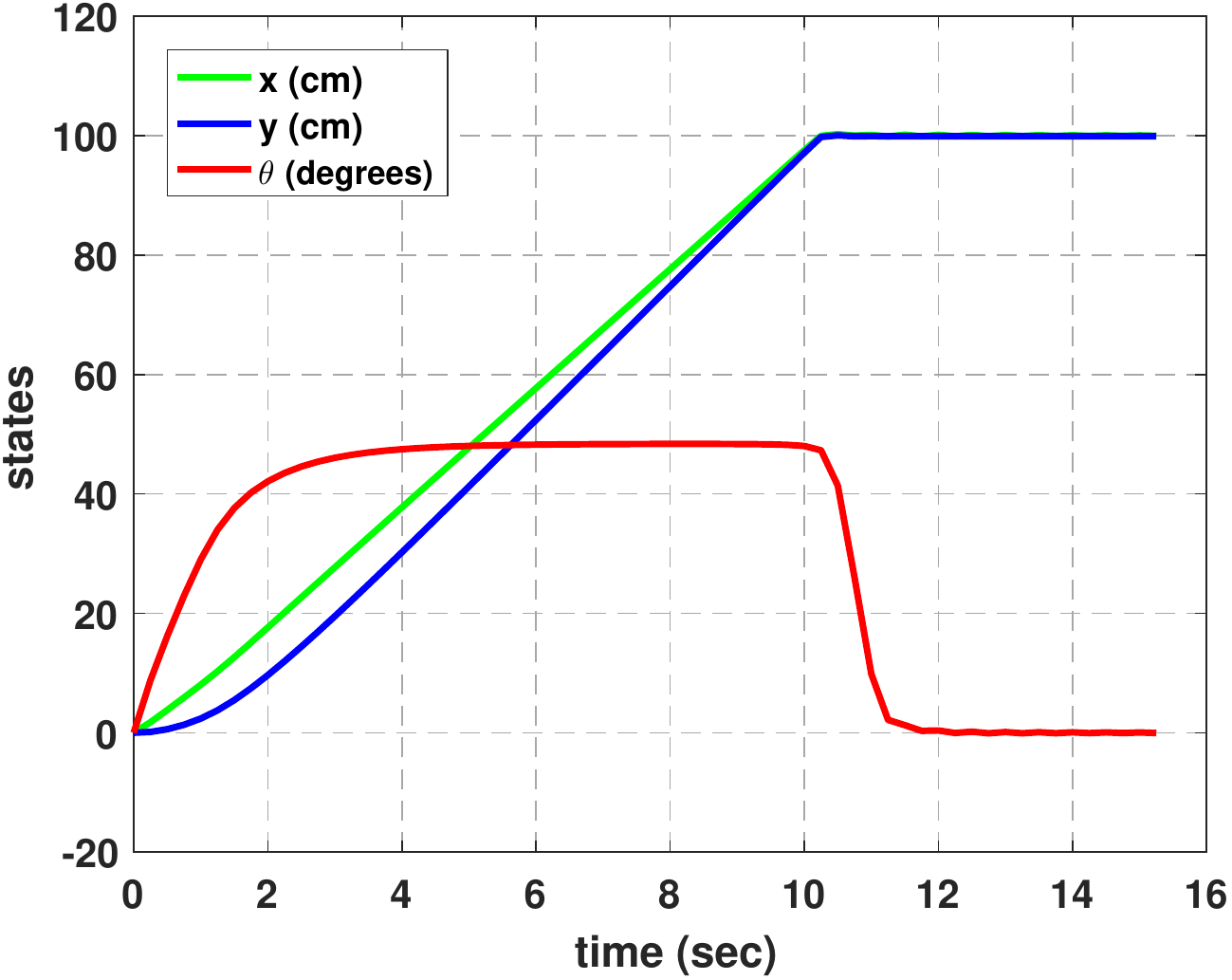} 
		\caption{} 
		\label{fig7:c} 
	\end{subfigure}
	\hfill%
	\begin{subfigure}[b]{0.45\linewidth}
		\centering
		\includegraphics[width=\linewidth,height = 0.15\textheight]{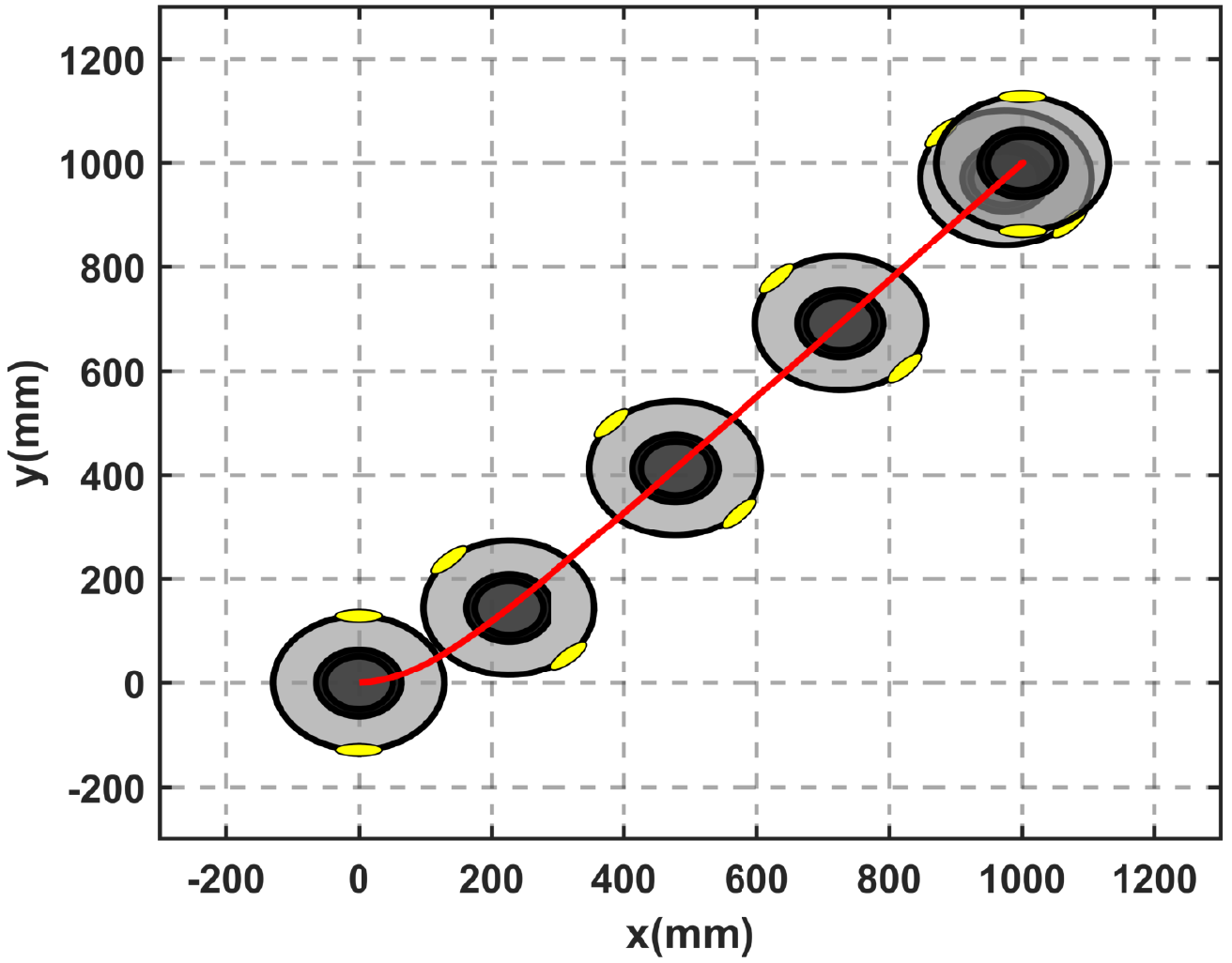} 
		\caption{} 
		\label{fig7:d} 
	\end{subfigure} 
	\caption{Differential drive using first- (top) and second-order (bottom) needle variation actions. Snapshots of the system are shown at $t = 0, 2.5, 5, 7.5, 10$, and $12.5$~sec. The target state is $[x_d, y_d, \theta_d] = [1000$~mm$,1000$~mm$,0]$.}
	\label{Differential Drive} 
\end{figure}

We also present a Monte Carlo simulation that compares convergence success using first- and second-order needle variations controls shown in \eqref{optimalu} and \eqref{optcon}, respectively. We sampled over initial coordinates $x_0, y_0 \in [-1500, 1500]$~mm using a uniform distribution and keeping only samples for which the initial distance from the origin exceeded $L/5$; $\theta_0 = 0$ for all samples. Successful samples were within $L/5$ from the origin with an angle $\theta <\pi/12$ within 60 seconds using feedback sampling rate of 4 Hz. Results were generated using $Q = \text{diag}(10,10,1000)$, $P_1 = \text{diag}(0,0,0)$, $T = 0.5$~s, $R = \text{diag}(100,100)$ for \eqref{optimalu}, $R = \text{diag}(0.1,0.1)$ for \eqref{optcon}, $\gamma = -15$, $\lambda = 0.1$ and saturation limits on the angular velocities of each wheel $\pm$150/36~mm/s\,\footnote{The metric on control effort is necessarily smaller for \eqref{optcon}, due to parameter $\lambda$. The parameter was chosen carefully to ensure that control solutions from \eqref{optcon} and \eqref{optimalu} were comparable in magnitude.}. As shown in Fig.~\ref{DifDriveMC_2nd}, the system always converges to the target using second-order needle variation actions, matching the theory. 

\begin{figure}[]
	\centering
	\includegraphics[width=0.5\linewidth, height = 0.15\textheight]{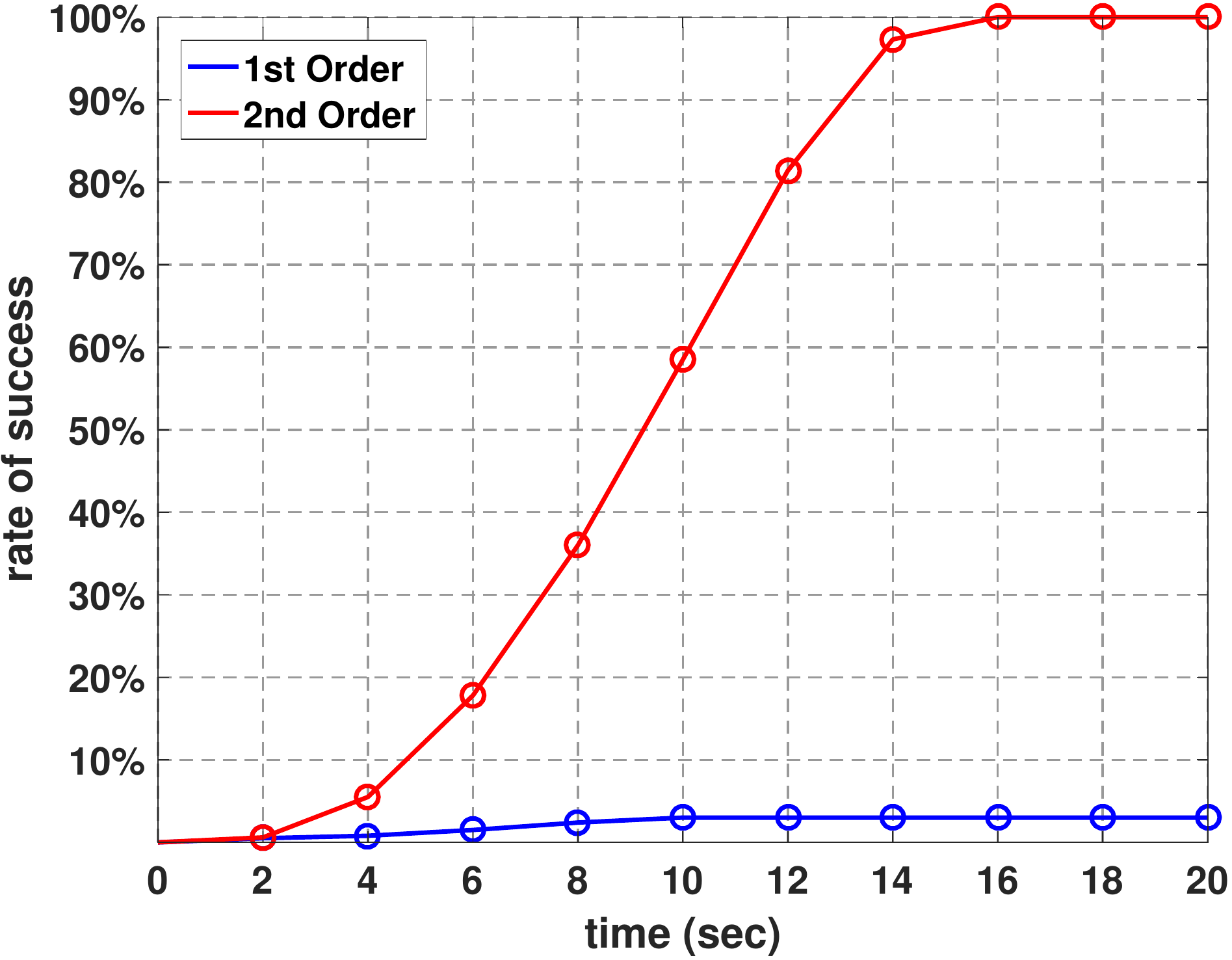}
	\caption{Convergence success rates of first- \eqref{optimalu} and second-order \eqref{optcon} needle variation controls for the kinematic differential drive model. Simulation runs: 1000.}\label{DifDriveMC_2nd}
\end{figure}

\subsection{3D Kinematic Rigid Body}
The underactuated kinematic rigid body is a three dimensional example of a system that is controllable with first-order Lie brackets. To avoid singularities in the state space, the orientation of the system is expressed in quaternions\cite{titterton2004strapdown, kuipers1999quaternions}. The states are $s = [x, y, z, q_0, q_1, q_2, q_3]$, where $b = [x, y, z]$ are the world-frame coordinates and $q = [q_0, q_1, q_2, q_3]$ are unit quaternions. Dynamics $f = [\dot{b},\dot{q}]^T$ are given by 

\begin{gather}
\dot{b} = R_qv, \label{dotb}\\
\dot{q} = \frac{1}{2}\begin{bmatrix} -q_1 & -q_2 & -q_3 \\ ~~q_0& -q_3& ~~q_2 \\ ~~q_3& ~~q_0& -q_1\\ -q_2& ~~q_1& ~~q_0 \end{bmatrix}\omega, \label{dotq}
\end{gather}
where $v$ and $\omega$ are the body frame linear and angular velocities, respectively\cite{da2015benchmark}. The rotation matrix for quaternions is
\medmuskip = 0.5mu
\begin{align*}
R_q=\begin{bmatrix} q_0^2 + q_1^2 - q_2^2 - q_3^2&  2(q_1q_2 - q_0q_3)& 2(q_1q_3+q_0q_2) \\
2(q_1q_2 + q_0q_3)&      q_0^2 - q_1^2+q_2^2 - q_3^2&  2(q_2q_3 - q_0q_1)\\
2(q_1q_3 - q_0q_2)&      2(q_2q_3 + q_0q_1)&      q_0^2-q1^2 -q_2^2+ q_3^2\end{bmatrix}.
\end{align*}

The system is kinematic: $v = F$ and $\omega = T$, where $F = (F_1, F_2, F_3)$ and $T = (T_1, T_2, T_3)$ describe respectively the surge, sway, and heave input forces, and the roll, pitch, and yaw input torques. We render the rigid body underactuated by removing the sway and yaw control authorities ($F_2 = T_3 = 0$). 

The four control vectors span a four dimensional space. First order Lie bracket terms add two more dimensions to span the state space ($\mathbb{R}^6$) (the fact that there are seven states in the model of the system is an artifact that is inherent in quaternion representation; it does not affect controllability, given that there is also a constraint that the norms of quaternions must sum up to one).
\begin{figure}[]
	\centering
	\includegraphics[width=0.5\linewidth, height = 0.15\textheight]{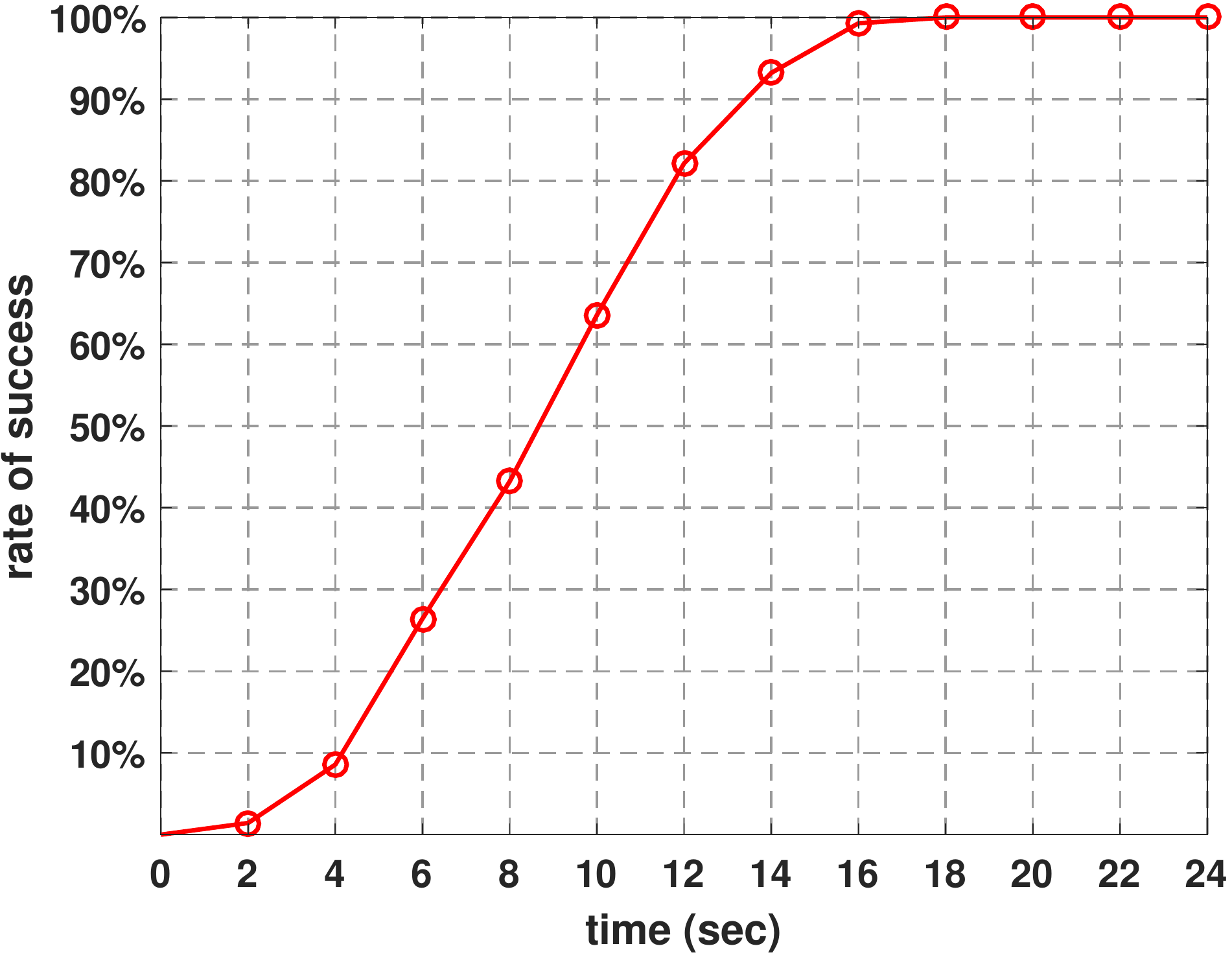}
	\caption{Convergence success rates of second-order needle variation controls \eqref{optcon} for the underactuated kinematic vehicle. First-order actions \eqref{optimalu} do not affect the $y$-coordinate of the rigid body and therefore never converge. Simulation runs: 280.}\label{KinMC_2nd}
\end{figure}
The vectors $h_1, h_2, [h_2, h_3]$ span $\mathbb{R}^3$ associated with the world frame coordinate dynamics $\dot{x}, \dot{y}$, and $\dot{z}$. Similarly, vectors $h_3, h_4$, and $[h_4, h_3]$ also span $\mathbb{R}^3$. Thereby, control vectors and their first-order Lie brackets span the state space and, from Theorem \ref{Theorem}, optimal actions shown in \eqref{optcon} will always reduce the cost function \eqref{cost}. 

To verify the theory, we present the convergence success of the system on 3D motion (see Fig.~\ref{KinMC_2nd}). Using Monte Carlo sampling with uniform distribution, initial locations were randomly generated such that $x_0, y_0, z_0 \in [-50, 50]$~cm keeping only samples for which the initial distance from the origin exceeded 6~cm. We regarded as a convergence success each trial in which the rigid body was within 6~cm to the origin by the end of 60 seconds at any orientation. Results were generated at a sampling rate of 20~Hz using $Q = 0$, $P_1 = \text{diag}(100,200,100,0,0,0,0)$, $T = 1.0$~s, $\gamma = -50000$, $\lambda = 10^{-3}$, $R~=~10^{-6}\,\text{diag}(1,1,100,100)$ for \eqref{optcon}, and $R = \text{diag}(10,10,1000,1000)$ for controls in \eqref{optimalu}. Controls were saturated at $\pm 10$~cm/s for the linear velocities and $\pm 10$~rad/s for the angular ones. As shown in Fig.~\ref{KinMC_2nd}, and as expected, all locomotion trials were successful. 

\subsection{Underactuated Dynamic 3D Fish}
\begin{figure}[]
	\centering
	\includegraphics[width=0.5\linewidth, height = 0.15\textheight]{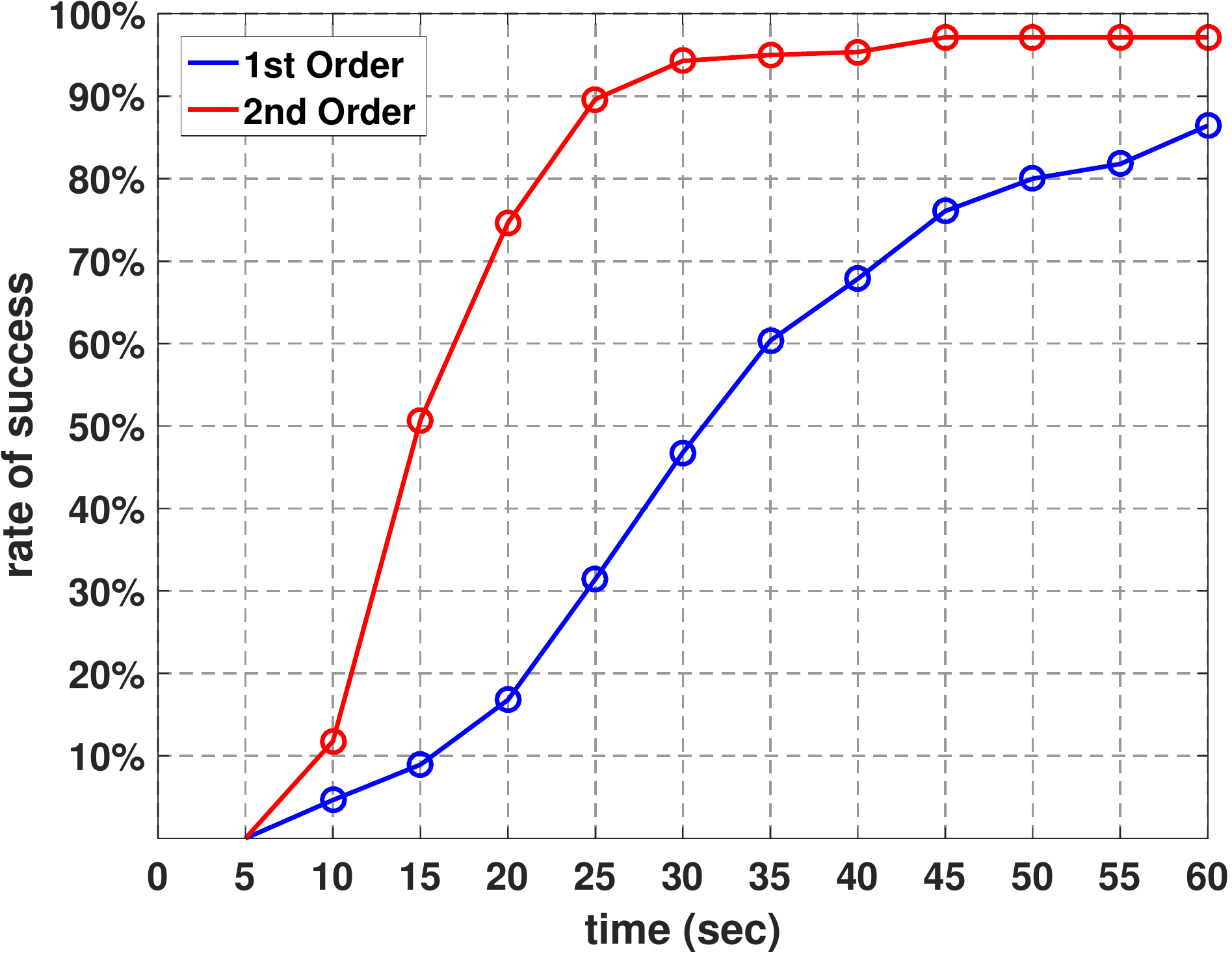}
	\caption{Convergence success rates of first- and second-order needle variation controls (\eqref{optimalu} and \eqref{optcon}, respectively) for the underactuated \textit{dynamic} vehicle model. Simulation runs: 280} 
	\label{DynMC}
\end{figure}
 \begin{figure*}%
	\centering
	\begin{subfigure}{\columnwidth}
		\includegraphics[width=0.8\columnwidth,height = 0.18\textheight]{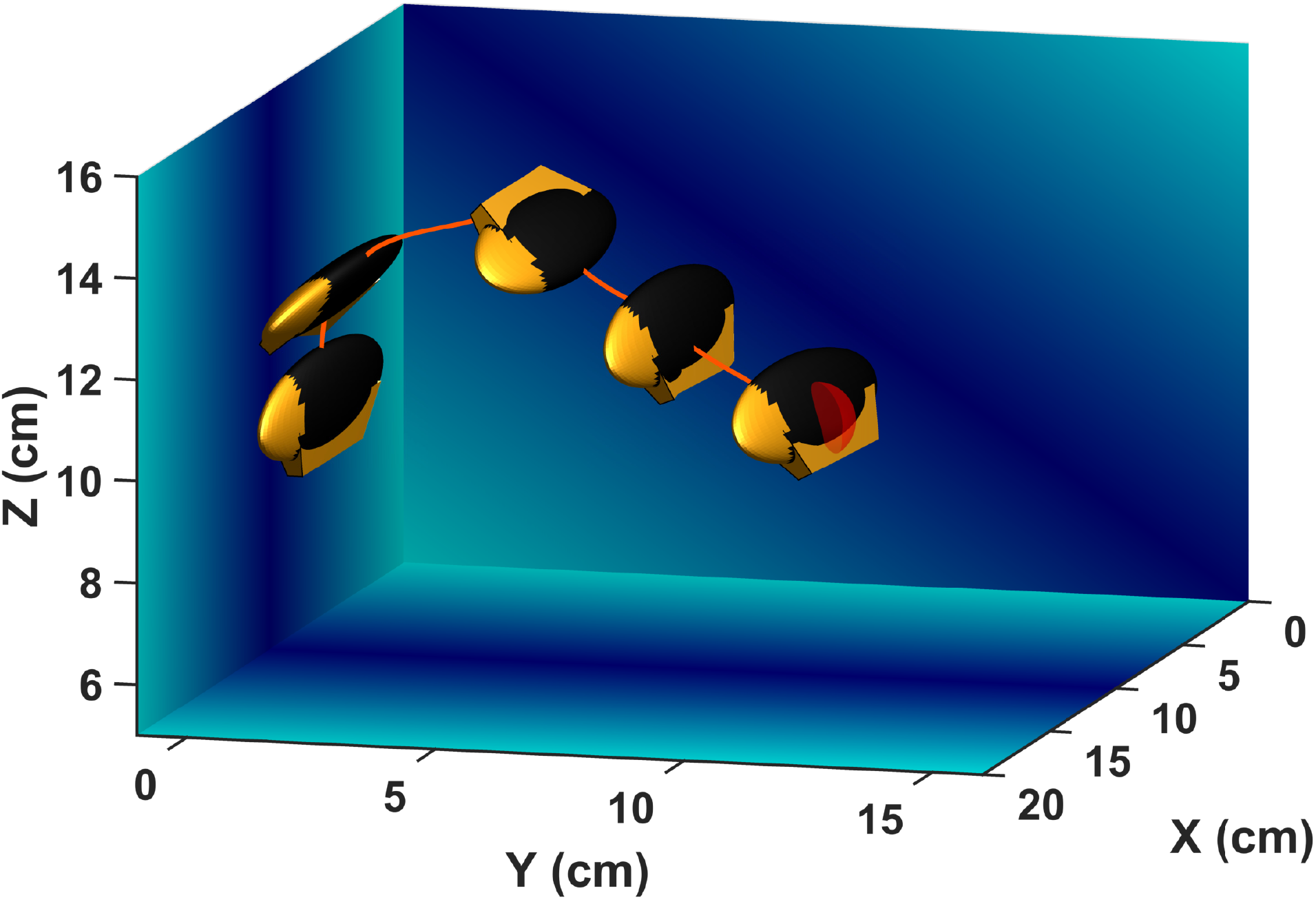}%
		\caption{}%
		\label{SidewayMovement}%
	\end{subfigure}\hfill%
	\begin{subfigure}{\columnwidth}
		\includegraphics[width=0.8\columnwidth,height = 0.18\textheight]{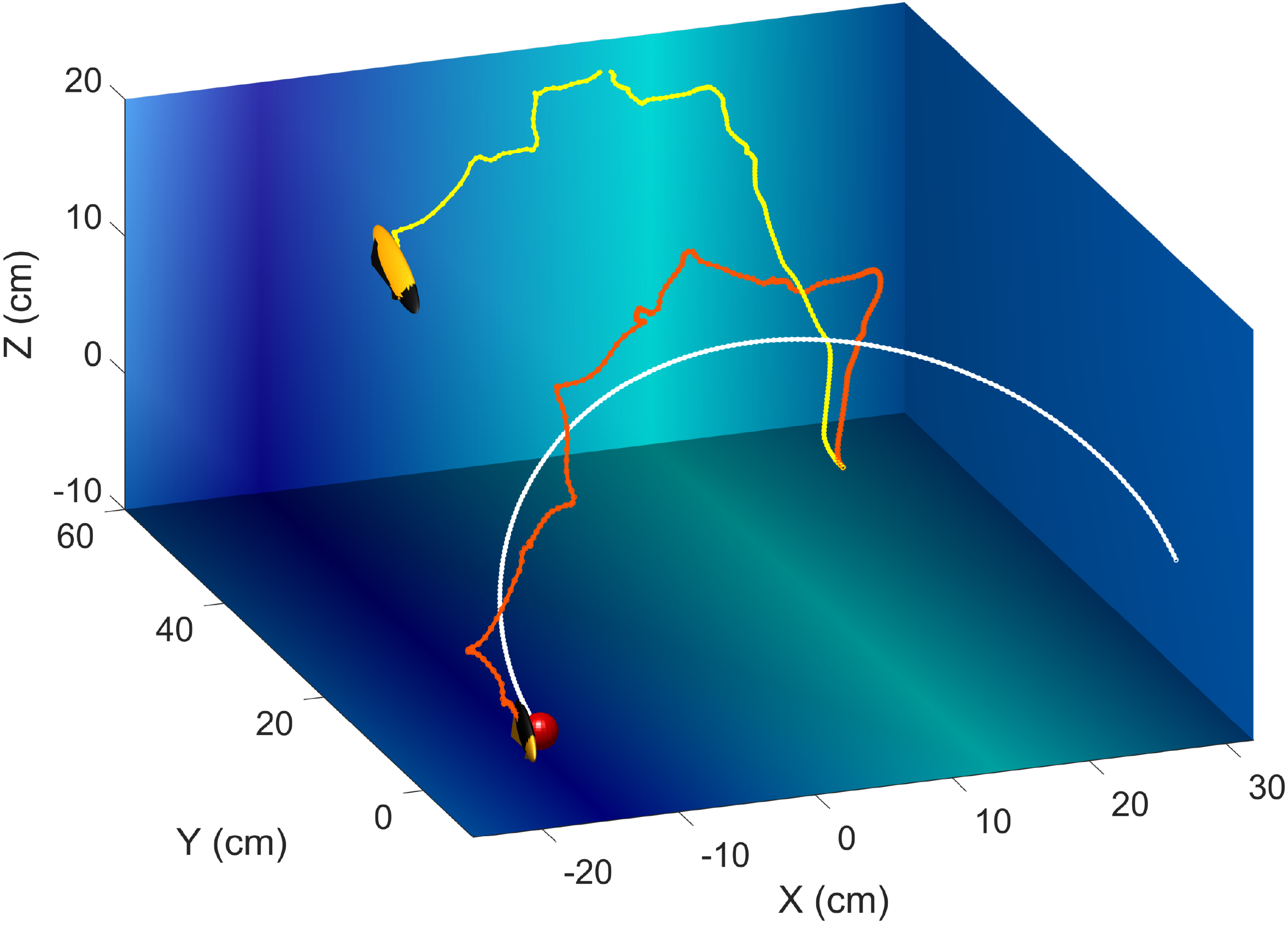}%
		\caption{}%
		\label{TrajTrack_drift}%
	\end{subfigure}\hfill%
	\caption{Figure \ref{SidewayMovement} shows snapshots of a parallel displacement maneuver using an underactuated dynamic vehicle model with second-order controls given by \eqref{optcon}; first-order solutions \eqref{optimalu} are singular throughout the simulation. Figure \ref{TrajTrack_drift} shows tracking performance of the same system in the presence of +10~cm/s $\hat{y}$ fluid drift. The yellow system corresponds to first-order needle variation actions; the red one to second order. The target trajectory (red ball) is indicated with white traces over a 10-second simulation. Animation of these results is available at https://vimeo.com/219628387.}
\end{figure*}
We represent the three dimensional rigid body with states $s~=~[b,~q,~v,~\omega]^T$, where $b = [x, y, z] $ are the world-frame coordinates, $q = [q_0,q_1, q_2, q_3]$ are the quaternions that describe the world-frame orientation, and $v = [v_x, v_y, v_z]$ and $\omega = [\omega_x, \omega_y, \omega_z]$ are the body-frame linear and angular velocities. 
The rigid body dynamics are given by $\dot{b}$ and $\dot{q}$ shown in \eqref{dotb} and \eqref{dotq} and
\begin{gather*}
M \dot{v} = Mv \times \omega + F, \\
J \dot{\omega} = J\omega \times \omega + T,
\end{gather*}
where the effective mass and moment of inertia of the rigid body are given by $M~=~\text{diag}(6.04, 17.31, 8.39)$~g and $J~=~\text{diag}(1.57, 27.78, 54.11)$g$\cdot$cm$^2$, respectively. This example is inspired by work in \cite{mamakoukas2016,postlethwaite2009optimal} and the parameters used for the effective mass and moment of inertia of a rigid body correspond to measurements of a fish. The control inputs are $F_2 = T_3 = 0$ and $F_3\ge0$.

The control vectors only span a four dimensional space and, since they are state-independent, their Lie brackets are zero vectors. However, the Lie brackets containing the drift vector field $g$ (that also appear in the MIH expression) add from one to four (depending on the states) independent vectors such that control solutions in \eqref{optcon} guarantee decrease of the cost function \eqref{cost} for a wider set of states than controls in \eqref{optimalu}. 

Simulation results based on Monte Carlo sampling are shown in Fig.~\ref{DynMC}. Initial coordinates $x_0, y_0, z_0$ were generated using a uniform distribution in $[-100, 100]$~cm, discarding samples for which the initial distance to the origin was less than 15~cm. Successful trials where the ones for which, within a simulation window of 60 seconds, the system approached within 5~cm to the origin (at any orientation) and whose magnitude of the linear velocities was, at the same time, less than 5~cm/s. Results were generated at a sampling rate of 20~Hz using \medmuskip=0mu
\thinmuskip=0mu
\thickmuskip=0mu$T = 1.5$~s, $P_1 = 0$, $Q~=~\frac{1}{200}\text{diag}(10^3,10^3,10^3,0,0,0,0, 1, 1, 1, 2\cdot10^3,10^3,10^3)$, $\gamma = -5$, $R = \text{diag}(10^3,10^3,10^6,10^6)$ for \eqref{optimalu}, $R = \frac{1}{2}\,\text{diag}(10^{-6},10^{-6}, 10^{-3},10^{-3})$ for \eqref{optcon}, and $\lambda = 10^{-4}$.\medmuskip=4mu
\thinmuskip=3mu
\thickmuskip=5mu ~The same control saturations ($F_1\in[-1, 1]$\,mN, $F_3\in[0,1]$\,mN, $T_1\in[-0.1, 0.1]$\,$\mu$N$\cdot$m, and $T_2\in[-0.1, 0.1]$\,$\mu$N$\cdot$m) were used for all simulations of the dynamic 3D fish. As shown in Fig. \ref{DynMC}, controls computed using second-order needle variations converge faster than those based on first-order needle variations, and 97\% of trials converge within 60 seconds. 

Both methods converge over time to the desired location; as the dynamic model of the rigid body tumbles around and its orientation changes, possible descent directions of the cost function \eqref{cost} change and the control is able to push the system to the target. Controls for the first-order needle variation case \eqref{optimalu} are singular for a wider set of states than second-order needle variation controls \eqref{optcon} and, for this reason, they benefit more from tumbling.
In a 3D parallel locomotion task, only second-order variation controls \eqref{optcon} manage to provide control solutions through successive heave and roll inputs, whereas controls based on first-order sensitivities \eqref{optimalu} fail (see Fig.~\ref{SidewayMovement}). 
\\
\indent As controls in \eqref{optcon} are non-singular for a wider subset of the configuration state space than the first-order solutions in \eqref{optimalu}, they will provide more actions over a period of time and keep the system closer to a time-varying target. Fig. \ref{TrajTrack_drift} demonstrates the superior trajectory tracking behavior of controls based on \eqref{optcon} in the presence of +10~cm/s $\hat{y}$ fluid drift. The trajectory of the target is given by $[x, y, z]$=\medmuskip=0mu
\thinmuskip=0mu
\thickmuskip=0mu[20\,+10\,cos($\frac{\text{t}}{5}$)\,cos($\frac{\text{3t}}{10}$), 20\,+\,10\,cos($\frac{\text{t}}{5}$)\,sin($\frac{\text{3t}}{10}$), 10\,sin($\frac{\text{2t}}{5}$)], with $T=2$~s, $\lambda=0.01$, $Q~=~\text{diag}(10,10,10,0,0,0,0, 0, 0, 0, 1,1,0.1)$, $\gamma=-50000$, $P_1=\text{diag}(10,10,10,0,0,0,0, 0, 0, 0, 0, 0, 0)$, $R=\text{diag}(10^3,10^3,10^6,10^6)$ for \eqref{optimalu}, and $R=\text{diag}(10,10, 10^4,10^4)$ for \eqref{optcon}.\medmuskip=4mu
\thinmuskip=3mu
\thickmuskip=5mu~The simulation runs in real time using a C++ implementation on a laptop with Intel$^\circledR$ Core$^{\text{TM}}$ i5-6300HQ CPU @2.30GHz and 8GB RAM.
 The drift is known for both first- and second-order systems and accounted for in their dynamics in the form of $\dot{b} = \dot{b} + \dot{b}_\text{drift}$, where $\dot{b}_\text{drift}$ is a vector that points in the direction of the fluid flow. Simulation results demonstrate superior tracking of second-order needle variation controls that manage to stay with the target, whereas, in the meantime, the system that corresponds to first-order needle variation controls is being drifted away by the flow. \\
 \indent We also tested convergence success of the +10~cm/s $\hat{y}$ drift case. Initial conditions $x,y,z$ were sampled uniformly from a $30$~cm radius from the origin, discarding samples for which the initial distance was less than $5$~cm. We consider samples to be successful if, during 60 seconds of simulation, they approached the origin within 5~cm. Out of 500 samples, controls based on second-order variations converged 91\% of the time (with average convergence time of 5.87~s), compared to 89\% for first-order actions (with average convergence time of 9.3~s). Simulation parameters were \medmuskip=0mu
	 \thinmuskip=0mu
	 \thickmuskip=0mu $T=1$~s, $\lambda=10^{-4}$, $Q=10^{-3}\text{diag}(10,10,10,0,0,0,0, 1, 1, 1, 1,1,1)$, $P_1=\text{diag}(100,100,100,0,0,0,0, \frac{1}{2}, \frac{1}{2}, \frac{1}{2}, 0, 0, 0)$, $\gamma=-25000$, $R=\text{diag}(0.1,0.1,10^4,10^4)$ for \eqref{optimalu}, and $R=\frac{1}{2}\text{diag}(10^{-5},10^{-5}, 1,1)$ for \eqref{optcon}.
\section{Conclusion} 
This paper presents a needle variation control synthesis method for nonlinearly controllable systems that can be expressed in control affine form. Control solutions provably exploit the nonlinear controllability of a system and, contrary to other nonlinear feedback schemes, have formal guarantees to decrease the objective. By optimally perturbing the system with needle actions, the proposed algorithm avoids the expensive iterative computation of controls over the entire horizon that other NMPC methods use and is able to run in real time. 

Simulation results on three underactuated systems compare first- and second-order needle variation controls and demonstrate the superior convergence success rate of the proposed feedback synthesis. Because second-order needle variation actions are non-singular for a wider set of the state space than controls based on first-order sensitivity, they are also more suitable for time-evolving objectives, as demonstrated by the trajectory tracking examples in this paper. Second order needle variation controls are also calculated at little computational cost and preserve control effort. These traits, demonstrated in the simulation examples of this paper, render feedback synthesis based on second- and higher-order needle variation methods a promising alternative feedback scheme for underactuated and nonlinearly controllable systems. 
\section*{Acknowledgments}
This work was supported by the Office of Naval Research under grant ONR N00014-14-1-0594. Any opinions, findings, and conclusions or recommendations expressed here are those of the authors and do not necessarily reflect the views of the Office of Naval Research.
\clearpage
\bibliographystyle{IEEEtran}
\balance
\bibliography{references}

\end{document}